\documentclass[12pt]{article}

\usepackage{amssymb,amsmath,amsthm,bbm,enumitem}
\usepackage[a4paper, total={6.3in, 9.2in}]{geometry}
\usepackage{setspace}
\usepackage{hyperref}

\newtheorem{theorem}{Theorem}[section]
\theoremstyle{definition}
\newtheorem{definition}{Definition}[section]
\theoremstyle{definition}

\theoremstyle{definition}
\newtheorem{proposition}{Proposition}[section]
\theoremstyle{definition}
\newtheorem{corollary}{Corollary}[section]
\theoremstyle{remark}
\newtheorem{remark}{Remark}[section]
\theoremstyle{definition}
\newtheorem{lemma}{Lemma}[section]
\DeclareMathOperator{\naturalnumber}{\mathbb{N}}
\DeclareMathOperator{\real}{\mathbb{R}}

\DeclareMathOperator{\p}{\mathbb{P}}
\DeclareMathOperator{\E}{\mathbb{E}}
\DeclareMathOperator{\borel}{\mathcal{B}}

\DeclareMathOperator{\as}{\text{-a.s.}}
\DeclareMathOperator{\1}{\mathbbm{1}}
\newcommand{\filtration}[2]{\mathcal{F}^{#1}_{#2}}
\newcommand{\density}[3]{\frac{d\mu_{ {#1} }}{ d\mu_{#2} }\left({#3} \right)}

\newcommand{\D}[2]{\mathbb{D}\left( {#1}\| {#2}   \right)   }
\setlist[itemize]{leftmargin=0.4cm,labelindent=\parindent}

\begin{document}
	
	\setstretch{1.2}
	
	\title{Relative Entropy Convergence under Picard's Iteration for Stochastic Differential Equations}
	
	\author{Tsz Hin, Ng\\ The University of Hong Kong \\ \href{mailto:kennthin@hku.hk}{kennthin@hku.hk} \and Guangyue, Han \\ The University of Hong Kong\\  \href{mailto:ghan@hku.hk}{ghan@hku.hk} } 	
	
	\date{\today}
	\maketitle
	
	\begin{abstract}
		For a family of stochastic differential equations, we investigate the asymptotic behaviors of its corresponding Picard's iteration, establishing convergence results in terms of relative entropy. Our convergence results complement the conventional ones in the $L^2$ and almost sure sense, revealing some previously unexplored aspects of the stochastic differential equations under consideration. For example, in combination with Pinsker's inequality, one of our results readily yields the convergence under Picard's iteration in the total variation sense, which does not seem to directly follow from any other known results. Moreover, our results promise possible further applications of SDEs in related disciplines. As an example of such applications, we establish the convergence of the corresponding mutual information sequence under Picard's iteration for a continuous-time Gaussian channel with feedback, which may pave way for effective computation of the mutual information of such a channel, a long open problem in information theory.
	\end{abstract}
	
	
	\section{Introduction}
	
	Let $T > 0$, and let $C[0, T]$ denote the space of all real-valued continuous functions over the interval $[0, T]$. Let $f$ be a {\em non-anticipative} functional from $[0,T] \times C[0, T] \times C[0, T]$ to $\real$ in the sense that, roughly speaking, with $t$ interpreted as the present time, the value of $f(t, \varphi, \psi)$ only causally depends on the ``history'' of $\varphi$ and $\psi$ (see the rigorous definition in Section~\ref{rigorous}). And let $g$ be a non-anticipative functional from $[0, T] \times C[0, T]$ to $\real$. In this paper, we are concerned with a stochastic differential equation (SDE) taking the following form:
	\begin{equation} \label{eq_sde}
		dY(t) = f(t,\xi,Y)dt + g(t,Y)dW(t), \ 0\leq t\leq T,
	\end{equation}
	with the initial condition $Y(0) = 0$, where $W=\{W(t); 0\leq t\leq T \} $ is a standard Brownian motion and $\xi = \{\xi(t); 0\leq t\leq T \}$ is a stochastic process independent of $W$ with $\E\left[ \sup_{0\leq t\leq T} |\xi(t)|^2 \right]<\infty$.
	
	For the special case that there is no actual dependence of $f$ on the second parameter $\varphi$, it is well known (see, e.g.,~\cite{BMSC, Mao}) that a convergence analysis of the corresponding Picard's iteration can establish the existence and uniqueness of the strong solution to the equation (\ref{eq_sde}), provided that $f$ and $g$ satisfy the conventional uniform Lipschitz and linear growth conditions. It turns out that, for the general case (that $f$ may actually depends on $\varphi$), the aforementioned conditions and convergence analysis can be adapted to achieve the same objective. More precisely, assume that $f$ and $g$ satisfy the following conditions:
	\begin{enumerate}
		\item[\textbf{(L)}]  \label{L} (Uniform Lipschitz) There exists $K>0$ such that for any $t\in[0,T]$ and any $\varphi,\psi,\tilde{\psi} \in C[0, T]$,
		\begin{equation}
			|f(t, \varphi, \psi)-f(t, \varphi, \tilde{\psi})|^2 + |g(t,\psi)-g(t,\tilde{\psi})|^2 \leq K\sup_{0\leq s\leq t}|\psi(s)-\tilde{\psi}(s)|^2.
		\end{equation}
		\item [\textbf{(G)}] \label{G}  (Linear Growth) There exists $L>0$ such that for any $t \in [0,T]$ and $\varphi,\psi \in \real$,
		\begin{equation}
			|f(t,\varphi,\psi)|^2 + |g(t,\psi)|^2 \leq L\left(1+\sup_{0\leq s\leq t}|\varphi(s)|^2 +\sup_{0 \leq s \leq t}|\psi(s)|^2\right).
		\end{equation}
	\end{enumerate}
	And consider the sequence $\{Y^{(n)}\}=\{Y^{(n)}(t); 0 \leq t \leq T\}_{n=0}^{\infty}$ recursively defined by
	\begin{equation} \label{eq_picard}
		Y^{(n)}(t)=  \int_0^t f\left(s, \xi, Y^{(n-1)}\right) ds + \int_0^t g\left(s,Y^{(n-1)}\right)dW(s), \mbox{ for $n \in \naturalnumber$},
	\end{equation}
	starting with $Y^{(0)}(t) \equiv 0$. Then, as elaborated in Theorem~\ref{thm_existence_uniqueness}, the sequence $\left\{ Y^{(n)} \right\}$ converges in the $L^2$ and almost-sure sense, which yields the unique strong solution $Y=\{Y(t); 0 \leq t \leq T \}$ to the equation (\ref{eq_sde}).
	
	In this paper, we will further characterize the convergence behaviors of $\left\{ Y^{(n)} \right\}$ in terms of the so-called {\em relative entropy}, a.k.a. {\em Kullback-Leibler divergence}. Roughly speaking, the relative entropy between two probability measures $\mu$ and $\nu$, denoted by $\D{\mu}{\nu}$, is a ``distance-like'' function measuring how ``close'' $\mu$ and $\nu$ are; and our main results show that as $n$ tends to infinity, the distribution of $\{Y^{(n)}\}$ converges to that of $Y$ with respect to certain measures in terms of relative entropy. More precisely, let $B=\{B(t);0\leq t\leq T \}$ be the strong solution to the following SDE:
	\begin{equation}
		\label{eq_B}
		dB(t) = g(t,B)dW(t) , \ 0\leq t\leq T,
	\end{equation}
	with the initial condition $B(0) = 0$. And let $\mu_{Y^{(n)}}$, $\mu_Y$ and $\mu_B$ denote the probability measures on the space of $C[0, T]$ induced by $Y^{(n)}$, $Y$ and $B$, respectively. Then, our main result can be stated as follows:
	\begin{theorem} \label{thm1}
		Suppose that $f$ and $g$ satisfy (\textbf{L}) and (\textbf{G}) and that $g$ satisfies the following regularity condition:
		\begin{enumerate}
			\item [(\textbf{R})] $\sup_{(t,\varphi)\in[0,T]\times C[0, T]}|g^{-1}(t,\psi)| <\infty.$
		\end{enumerate}
		Then, we have
		\begin{enumerate}
			\item[1)] $\lim_{n\to\infty} \D{\mu_{Y^{(n)}}}{\mu_Y}=0$;
			\item[2)] $\lim_{n\to\infty} \D{\mu_{Y^{(n)}}}{\mu_B} = \D{\mu_{Y}}{\mu_B}.$
		\end{enumerate}
	\end{theorem}
	
	Though often intuited as a way of measuring the distance between two probability distributions, relative entropy does not qualify as a metric due to its asymmetry (with respect to its two parameters) and failure to obey the triangle inequality. Nevertheless, relative entropy is intimately related to many other metrics such as total variation distance, Fisher information divergence, Wasserstein distance and so on; see, e.g.,~\cite{Villani2003, Villani2009} and references therein. This, together with many other desirable properties, explains its widespread use in diverse disciplines including probability theory~\cite{Oliver}, statistics~\cite{Kullback, Tsybakov}, statistical physics~\cite{Tribus, Hobson}, machine learning and neural science~\cite{Bishop, ChenJiaWang, cliffs, MacKay}, and information theory~\cite{CoverThomas, ITCS}. Complementing the convergence of Picard's iteration in the conventional $L^2$ and almost sure sense, relative entropy convergence as characterized in Theorem~\ref{thm1} may reveal aspects of SDEs unexplored as yet, which hold out the promise of new applications of SDEs in some uncharted research territories.

Roughly speaking, the well-known Pinsker's inequality~\cite{CsiszarKorner} states that the relative entropy between two probability measures dominates their total variation distance. As an example of the possible applications of our results, we note that the statement 1) of Theorem~\ref{thm1}, in combination with Pinsker's inequality, immediately leads to the following corollary, which does not seem to directly follow from any other known results.
\begin{corollary}
	Suppose $f$ and $g$ satisfy $(\bf{L})$, $(\bf{G})$ and $(\bf{R})$. Then, $\mu_{Y^{(n)}}$ converges to $\mu_Y$ in total variation distance.
\end{corollary}

For another example, the equation (\ref{eq_sde}) is often used to model a continuous-time Gaussian channel with feedback in information theory. As a corollary of the statement 2) of Theorem~\ref{thm1}, we establish the convergence of the mutual information of such a channel under Picard's iteration in Theorem~\ref{pp_mi}. Here, we note that this convergence result requires rather mild assumptions and more importantly feature a finer characterization of the convergence behavior, and thereby may pave way for effective computation of the mutual information of a continuous-time Gaussian channel with feedback, which is a long open problem in information theory.

		The remainder of the paper is organized as follows. We start with Section~\ref{sec_prelim}, which will introduce some basic notions and results that are instrumental in our proofs. The proof of our main result will then be presented in Section~\ref{Proof-Main-Result}. In Section~\ref{sec_proof}, we state and prove Theorem~\ref{pp_mi}, which exemplifies the possible applications of Theorem~\ref{thm1} in information theory.
	
	\section{Preliminaries} \label{sec_prelim}
	
	\subsection{Notations} \label{rigorous}
	
	We denote by $\naturalnumber$ and $\naturalnumber_0$ the sets of positive and non-negative integers respectively. For any $\varphi \in C[0, T]$, let $\|\varphi\|_T$ denote its sup-norm, that is, $\|\varphi\|_T := \sup_{0\leq t\leq T}|\varphi(t)|$. We will equip the space $C[0, T]$ with the filtration $\{\borel_t\}_{0\leq t\leq T}$, where $\borel_T$ denotes the Borel $\sigma$-algebra on the space $C[0, T]$ and $\borel_t = \pi_t^{-1}(\borel_T)$, where $\pi_t : C[0, T] \to C[0, T]$ is given by the map $(\pi_tx)(s) = x(t\wedge s)$. We say that a functional $f$ from $[0,T] \times C[0, T] \times C[0, T]$ to $\real$ is {\em non-anticipative} if $f$ is $\borel([0,T]) \otimes \borel_T \otimes \borel_T$-measurable and for each $t\in[0,T]$, $f(t,\cdot,\cdot)$ is $\borel_{t}\otimes\borel_{t}$-measurable, or equivalently, for any $\varphi, \psi \in C[0, T]$,
$$
f(t, \varphi, \psi)= f(t, \{\varphi(s); 0 \leq s \leq t\}, \{\psi(s); 0 \leq s \leq t\}).
$$
	
	We use $(\Omega,\mathcal{F},\p)$ to denote the underlying probability space equipped with a filtration $\mathcal{F}=\{\mathcal{F}_t \}_{0\leq t\leq T}$, which, as is typical in the theory of SDEs, satisfies the {\em usual conditions}~\cite{BMSC} and is rich enough to accommodate $\xi$ and $W$. For any $\mathcal{F}$-adapted process $X=\{X(t), \mathcal{F}_t;0\leq t\leq T\}$, we use $\mathcal{F}^X =\{\filtration{X}{t};0\leq t\leq T \}$ to denote the {\em augmented filtration} generated by $X$.
	
	For any two probability measures $\mu$ and $\nu$, we write $\mu\ll\nu$ to mean $\mu$ is absolutely continuous with respect to $\nu$ and $\mu\sim\nu$ if they are equivalent. For any two processes $X=\{X(t); 0 \leq t \leq T\}$ and $Y=\{X(t); 0 \leq t \leq T\}$, as mentioned before, we use $\mu_X, \mu_Y$ to denote the probability distributions on $\borel_T$ induced by $X, Y$, respectively; and  we write the Radon-Nikodym derivative of $\mu_Y$ with respect to $\mu_X$ as $d\mu_Y/d\mu_X$.

	Given two expressions $A_1, A_2$, we use $A_1 \lesssim_M A_2$ as a shorthand for the inequality $A \leq C_M B$ for some constant $C_M$ depending only on $M$.

	\subsection{The $L^2$ and Almost Sure Convergence of $\{Y^{(n)}\}$}
	
	The following theorem establishes the convergence of $\{Y^{(n)}\}$ in the $L^2$ and almost sure sense, and thereby the existence and uniqueness of the strong solution to (\ref{eq_sde}). The proof is a slight modification of the conventional argument and is included in Appendix~\ref{proof_thm_existence_uniqueness} for completeness. Here, we emphasize that the constants specified in the theorem will be used to characterize the rate of convergence in Theorem~\ref{pp_mi}.
	\begin{theorem} \label{thm_existence_uniqueness}
		Consider the equation (\ref{eq_sde}) satisfying the conditions (\textbf{L}) and (\textbf{G}). Then, we have
		\begin{enumerate}
			\item[1)] for all $n \in \naturalnumber_0$,
			\begin{equation} \label{eq_exist_unique_bound_picard}
				\E\left[\left\|Y^{(n)} \right\|_T^2\right] \leq k_1e^{k_2T} \quad \text{and} \quad  \E\left[\left\| Y^{(n+1)}-Y^{(n)} \right\|^2_T  \right] \leq c_1 \frac{(c_2T)^n}{n!},
			\end{equation}
			where $k_1 = 2LT(T+4)\left(1+\E\left[\|\xi\|_T^2\right] \right) $, $k_2 = 2L(T+4)$,  $c_1 =2TL(T+4) \left( 1+\E\left[\|\xi\|_T^2\right]\right)  $ and $c_2 =2K(T+4)$;
			\item[2)] there exists a unique strong solution $Y=\{Y(t);0\leq t\leq T \}$ to (\ref{eq_sde}) with
			\begin{equation} \label{eq_square_bound_sq}
				\lim_{n\to\infty}\left\|Y^{(n)}-Y  \right\|^2_T = 0 \mbox { and } \E\left[\left\| Y^{(n)}-Y \right\|^2_T  \right] \leq c_3\frac{(c_2T)^n}{n!} \stackrel{n \to \infty}{\longrightarrow} 0,
			\end{equation}
			where $c_3 =k_1e^{k_2T} $. So, the solution $Y$ is the pointwise and $L^2$-limit of the sequence $\left\{Y^{(n)} \right\}_{n=0}^\infty$.
		\end{enumerate}				
	\end{theorem}
	
	\subsection{Absolute Continuity between Relevant Measures} \label{ACRM}
	
	In this section, assuming $g \equiv 1$, we will discuss the absolute continuity between various measures and derive some formulas for the corresponding Radon-Nikodym derivatives which will be used in our proofs. For notational convenience, we may simply write $f\left(t,\xi, Y^{(n)} \right)$ as $f^{(n)}(t)$, and $f(t,\xi,Y)$ as $f(t)$.
	
	First of all, it is easy to see that the linear growth condition (\textbf{G}) guarantees that
	\begin{equation}
		\int_0^T |f(t,\varphi,\psi)|dt < \infty
	\end{equation}
	for any $\varphi,\psi \in C[0, T]$, and
	\begin{equation} \label{eq_cond_derivative1}
		\p\left( \int_0^T f^2\left(t,\xi,Y\right)dt <\infty \right)=1 \quad \text{and} \quad \p\left( \int_0^T \E^2\left[f\left(t,\xi,Y\right) \big| \filtration{Y}{t} \right] dt<\infty \right)=1.
	\end{equation}
	Hence, by~\cite[Theorem 7.14 and Lemma 7.7]{liptser}, we have
\begin{equation} \label{YsimW}
\mu_Y \sim \mu_W \mbox{  and  } \mu_{\xi,Y} \sim \mu_{\xi} \times \mu_W,
\end{equation}
    and moreover,
	\begin{equation}
		\label{eq_density_Y,W}
		\frac{d\mu_Y}{d\mu_W}(Y) = \exp \left(\int_0^T \E\left[f(t,\xi,Y) |\filtration{Y}{t} \right]dY(t) - \frac{1}{2} \int_0^T \E^2\left[f(t,\xi,Y) |\filtration{Y}{t} \right]dt    \right)
	\end{equation}
	and
	\begin{equation}
		\frac{d\mu_{\xi,Y}}{d(\mu_\xi\times\mu_W)}(\xi,Y) = \exp\left(\int_0^T f(t,\xi,Y)dY(t) -\frac{1}{2}\int_0^T f^2(t,\xi,Y)dt \right).
	\end{equation}
	
	Similarly, it follows from (\textbf{G}) and (\ref{eq_square_bound_sq}) that for any $n\in\naturalnumber_0$,
	\begin{equation}
		\p\left( \int_0^T f^2\left(t,\xi,Y^{(n)}\right)dt <\infty \right)=1 \quad \text{and} \quad \p\left( \int_0^T \E^2\left[f\left(t,\xi,Y^{(n)}dt\right) \big| \filtration{Y^{(n+1)}}{t} \right]dt <\infty \right)=1.
	\end{equation}
	Hence, by~\cite[Theorem 7.14]{liptser}, we have $\mu_{Y^{(n)}} \ll \mu_W$ for all $n\in\naturalnumber$, and furthermore,
	\begin{align} \label{eq_logderivative_uncondition}
		\hspace{-5mm} \frac{d\mu_{Y^{(n)}}}{d\mu_W}\left(Y^{(n)} \right) &= \exp\left(\int_0^T\alpha^{(n)}(t) dY^{(n)}(t) - \frac{1}{2}\int_0^T  \left(\alpha^{(n)}(t)\right)^2 dt \right) \nonumber \\
		&\hspace{-2cm} = \exp\left(\int_0^T \alpha^{(n)}(t) f^{(n-1)}(t) dW(t) + \int_0^T f^{(n-1)}(t) \alpha^{(n)}(t) dt  - \frac{1}{2}\int_0^T  \left(\alpha^{(n)}(t)\right)^2 dt  \right),
	\end{align}
	where $\alpha^{(n)}(t) := \E\left[f\left(t,\xi,Y^{(n-1)}\right) \big| \filtration{Y^{(n)}}{t}\right]$, and, as mentioned before, $f^{(n)}(t)$ is a shorthand notation for $f(t, \xi, Y^{(n)})$. The following proposition, whose proof is postponed to Appendix~\ref{proof_lemma_novikov}, says that an extra condition will guarantee that $\mu_{Y^{(n)}}$ and $\mu_W$ are equivalent.
	\begin{proposition} \label{lemma_novikov}
		If $f$ satisfies (\textbf{L}) and there exists an $\varepsilon>0 $ such that
		\begin{equation} \label{eq_cond_message}
			\E\left[e^{\varepsilon\|\xi \|^2_T } \right]<\infty,
		\end{equation}
		then $\mu_{Y^{(n)}} \sim \mu_W$.
	\end{proposition}
	
	We next show the absolute continuity of $\mu_{\xi,Y^{(n)}}$ with respect to $\mu_\xi \times \mu_W$ and derive the formula for $\dfrac{d\mu_{\xi,Y^{(n)}}}{d\left(\mu_\xi\times\mu_W \right)} \left(\xi,Y^{(n)}\right)$, for which we will extend the results in~\cite[Lemmas 7.6 and 7.7]{liptser} below. Consider the process $Z=\{Z(t); 0 \leq t \leq T\}$ defined by
	\begin{equation}
		\label{eq_girs_y}
		Z(t) = \int_0^t h(s, \xi, W)ds + W(t),  \ 0\leq t\leq T,
	\end{equation}	
	where $h: [0, T] \times C[0, T] \times C[0, T] \to \mathbb{R}$ is a {\em non-anticipative} functional, and $\xi$ and $W$ are as in (\ref{eq_sde}). Then, the extension (see Appendix~\ref{proof_pp_girs2} for the proof) says
	\begin{proposition} \label{pp_girs2}
		Suppose that
		\begin{equation} \label{eq_cond_girs2}
			\E\left[\int_0^T h^2(t,\xi,W)dt \right]<\infty, \mbox{ and for any } x, w \in C[0, T], \int_0^T |h(t,x,w)|dt < \infty.
		\end{equation}	
		Then, $\mu_{\xi,Z} \ll \mu_\xi \times \mu_Z$ and moreover,
		{\small \begin{equation}
				\hspace{-1cm} \frac{d\mu_{\xi,Z}}{d(\mu_\xi\times \mu_W)}(\xi,Z) = \exp\left(\int_0^T \E\left[h(t,\xi,W) \big| \filtration{Z,\xi}{t}  \right] dZ(t) - \frac{1}{2}\int_0^T \E^2\left[h(t,\xi,W) \big| \filtration{Z,\xi}{t}  \right] dt \right).
		\end{equation}}
	\end{proposition}
	\noindent Note that an application of Proposition \autoref{pp_girs2} with $h(t, \xi, W) = f(t, \xi, Y^{(n)})$ immediately yields that $\mu_{\xi,Y^{(n)}}$ is absolutely continuous with respect to $\mu_\xi\times\mu_W$, and moreover,
	\begin{equation}
		\label{eq_derivative_step1}
		\frac{d\mu_{\xi,Y^{(n)}}}{d\left(\mu_\xi\times\mu_W \right)} \left(\xi,Y^{(n)}\right) = \exp\left( \int_0^T \beta^{(n)}(t)dY^{(n)}(t)-\frac{1}{2} \int_0^T \left(\beta^{(n)}(t) \right)^2 dt   \right),
	\end{equation}
	where $\beta^{(n)}(t) :=\E\left[f \left(t,\xi,Y^{(n-1)}\right)  \big| \filtration{\xi,Y^{(n)}}{t} \right]$.

	\subsection{Relative Entropy}
	
	\begin{definition}[Relative entropy]
		For two probability measures $\mu$ and $\nu$ on a same probability space, the {\em relative entropy} of $\mu$ with respect to $\nu$ is defined as
		\begin{equation*}
			\D{\mu}{\nu} = \begin{cases}
				\int \log\frac{d\mu}{d\nu}(x)\mu(dx), &\text{if $\mu\ll\nu$;}\\
				\infty &\text{otherwise,}
			\end{cases}
		\end{equation*}
		where $d\mu/d\nu(x)$ is the Radon-Nikodym derivative of $\mu$ with respect to $\nu$.
	\end{definition}
	
	It is well known that relative entropy $\D{\mu}{\nu}$ is convex in $(\mu,\nu)$. To be exact, for any $\alpha \in[0,1]$ and measures $\mu_i, \nu_i, \ i=1,2$, we have
	$$
	\D{\mu}{\nu} \leq \alpha \D{\mu_1}{\nu_1} + (1-\alpha) \D{\mu_2}{\nu_2},
	$$
	where $\mu = \alpha\mu_1 + (1-\alpha) \mu_2$ and $\nu = \alpha\nu_1 + (1-\alpha) \nu_2$. It turns out, as elaborated in the following lemma (see Appendix~\ref{convexity-lemma} for the proof), the convexity can be generalized to our setting, which is crucial in our proof of Theorem~\ref{thm1}.
	\begin{proposition} \label{lemma_relative_entropy_convex}
		Let $U,V$ and $Z$ be $C[0, T]$-valued random variables defined on the probability space $(\Omega,\mathcal{F},\p)$. Suppose $\mu_U \ll \mu_V$. Then,
		\begin{equation}
			\D{\mu_U}{\mu_V} \leq \int_{C[0, T]} \D{\mu_{U,Z}(\cdot,z) }{\mu_{V,Z}(\cdot,z)} \mu_Z(dz).
		\end{equation}
	\end{proposition}
	
	\section{Proof of \autoref{thm1}} \label{Proof-Main-Result}
	
	First of all, we show that the proof boils down to that for the special case $g \equiv 1$.
	
	Consider the process $\tilde{Y} =\{\tilde{Y}(t); 0\leq t\leq T \}$ satisfying the following SDE:
	\begin{equation} \label{eq_modified_Y}
		d\tilde{Y}(t) = \tilde{f}(t,\xi,\tilde{Y})dt + dW(t), \ 0 \leq t \leq T,
	\end{equation}
	with the initial condition $Y(0) = 0$, where $\tilde{f} := fg^{-1}$. It follows from (\textbf{L}), (\textbf{G}) and (\textbf{R}) that $\tilde{f}$ satisfies (\textbf{L}) and (\textbf{G}), and hence the solution to (\ref{eq_modified_Y}) uniquely exists. Then, by~\cite[Theorem 7.19]{liptser}, we have
	\begin{equation}
		\label{eq_y_tildeY}
		\density{Y}{B}{y} = \density{\tilde{Y}}{W}{y} \quad \text{and} \quad  \frac{d\mu_{\xi,Y}}{d\left(\mu_\xi \times \mu_B  \right)}(x,y) = \frac{d\mu_{\xi,\tilde{Y}}}{d\left(\mu_\xi \times\mu_W  \right)}(x,y) ,
	\end{equation}
	where $B=\{B(t);0\leq t\leq T\}$ is the process defined in (\ref{eq_B}). Note that the Picard's iteration corresponding to (\ref{eq_modified_Y}) recursively computes: for any $n \in \mathbb{N}$,
	\begin{equation} \label{eq_picard3}
		\tilde{Y}^{(n)}(t) =  \int_0^t \tilde{f}\left(s,\xi,\tilde{Y}^{(n-1)}\right)ds + W(t),
	\end{equation}
	starting with $\tilde{Y}^{(0)} \equiv 0$. Therefore, following the proofs of~\cite[Theorems 7.14, 7.18 and 7.19]{liptser}, we have
	\begin{equation} \label{eq_yn_tildeYn}
		\density{Y^{(n)}}{B}{y} = \density{\tilde{Y}^{(n)}}{W}{y} \quad \text{and} \quad  \frac{d\mu_{\xi,Y^{(n)}}}{d\left(\mu_\xi \times\mu_B  \right)}(x,y) = \frac{d\mu_{\xi,\tilde{Y}^{(n)}}}{d\left(\mu_\xi \times\mu_W  \right)}(x,y) .
	\end{equation}
	It then immediately follows from (\ref{eq_y_tildeY}) and (\ref{eq_yn_tildeYn}) that Theorem~\ref{thm1} established for the special case $g \equiv 1$ implies that for the general case. So, throughout the remainder of the proof, we will assume that $g \equiv 1$.
	
	\subsection{Some Useful Relations} \label{useful-relations}
	
	For each $x \in C[0, T]$, consider the process $Y^x=\{Y^x(t); 0\leq t\leq T \}$ satisfying the following SDE:
	\begin{equation}
		\label{eq_gaussian_fixedmessage}
		dY^{x}(t) = f(t,x,Y^x)dt + dW(t) , \ 0\leq t\leq T,
	\end{equation}
	with the initial condition $Y^x(0) = 0$. The corresponding Picard's iteration recursively computes: for any $n \in \mathbb{N}$,
	\begin{equation} \label{eq_picard_fixed}
		Y^{x,(n)}(t) =  \int_0^t f\left(s,x,Y^{x,(n-1)}\right)ds + W(t),
	\end{equation}
	starting with $Y^{(0)}(t) \equiv 0$. As discussed in Section~\ref{ACRM}, it holds that $\mu_{Y^x}  \sim \mu_W$ and furthermore
	\begin{equation} \label{eq_density_Yx,W}
		\density{Y^x}{W}{Y^x} = \exp \left( \int_0^T f(t,x,Y^x)dY^x(t) - \frac{1}{2} \int_0^T f^2(t,x,Y^x) dt \right).
	\end{equation}
	
	We now claim that, for any $n \in \naturalnumber_0$, there exists a non-anticipative functional $Q^{(n)}:[0,T]\times C[0, T] \times C[0, T] \to \mathbb{R}$ such that for any $x \in C[0, T]$,
	\begin{equation}  \label{conditional-on-x}
		\p\left(Y^{(n)}(t) = Q^{(n)}(t,\xi,W); \ 0\leq t\leq T \right) = \p\left( Y^{x,(n)}(t) = Q^{(n)}(t,x,W); \ 0\leq t\leq T\right)=1.
	\end{equation}
	Indeed, for $n=0$, one can simply set $Q^{(n)} \equiv 0$. By way of induction, suppose that (\ref{conditional-on-x}) is satisfied by $Q^{(n)}$ for a particular $n$. Then, we have
	\begin{align}
		Y^{x,(n+1)}(t)&=\int_0^t f\left(s,x,Y^{x,(n)}\right)ds + W(t) \nonumber \\
		&= \int_0^t f\left(s,x,Q^{(n)}(\cdot,x,W)\right)ds + W(t).
	\end{align}
	Setting $Q^{(n+1)}(t,x,y):=\int_0^t f\left(s,x,Q^{(n)}(\cdot,x,y)\right)ds + y(t)$, one readily verifies that for any $t\in[0,T]$,
	$$ Y^{(n+1)}(t) = Q^{(n+1)}(t,\xi,W) \ \text{and}\ Y^{x,(n+1)}(t) = Q^{(n+1)}(t,x,W) \ \p\as$$
	Since the above processes admit continuous sample paths, (\ref{conditional-on-x}) holds true for $Q^{(n+1)}$, which in turn yields the claim.
	
	Note that when $\xi$ is fixed as $x$, (\ref{eq_cond_message}) is trivially satisfied, and therefore, by Proposition~\ref{lemma_novikov}, it holds that $\mu_{Y^{x,(n)}} \sim \mu_W$, and, by Proposition~\ref{pp_girs2},
    \begin{equation} \label{den-1}
		\density{Y^{x,(n)}}{W}{Y^{x, (n)}}  = \exp\left(\int_0^T \Phi^{(n)}\left(t,Y^{x, (n)}\right)dY^{x, (n)}(t) - \frac{1}{2}\int_0^T\left(\Phi^{(n)}\left(t,Y^{x, (n)}\right)\right)^2  dt \right),
	\end{equation}
	where $\Phi^{(n)}: [0,T]\times C[0, T] \to \real$ is a non-anticipative functional such that for any $t \in [0, T]$,
	\begin{equation} \label{eq_pp_conditional}
		\Phi^{(n)}\left(t,Y^{x,(n)}\right) = \E\left[f\left(t,x,Y^{x,(n-1)}\right) \big| \filtration{Y^{x,(n)}}{t} \right] \ \p\as.
	\end{equation}

Finally, mimicking the proof of \cite[Lemma 7.6]{liptser} with the representation as in (\ref{conditional-on-x}), we can show that for all $x \in C[0, T]$ and $n \in \naturalnumber_0$,
	\begin{equation} \label{eq_relate_fixed_notfixed1}
		\frac{d\mu_{\xi,Y^{(n)}}}{d\left( \mu_\xi \times \mu_W  \right)}\left(x,Y^{(n)}\right) = \frac{d\mu_{Y^{x,(n)}}}{d\mu_W}\left(Y^{(n)} \right) \ \p\as,
	\end{equation}
	and
	\begin{equation} \label{eq_relate_fixed_notfixed2}
		\frac{d\mu_{\xi,Y}}{d\left( \mu_\xi \times \mu_W  \right)}(x,Y) = \frac{d\mu_{Y^{x}}}{d\mu_W}(Y) \ \p\as.
	\end{equation}
	And moreover, by~\cite[Lemma 7.7]{liptser}, we have the following relations
	\begin{equation}
		\label{eq_relate_cond_uncond1}
		\frac{d\mu_{Y^{(n)}}}{d\mu_W}\left(Y^{(n)}\right) = \int_{C[0, T]}\frac{d\mu_{\xi,Y^{(n)}}}{d(\mu_\xi\times\mu_W)}\left(x,Y^{(n)}  \right) \mu_\xi(dx) \ \ \p \as,
	\end{equation}
	and
	\begin{equation}
		\label{eq_relate_cond_uncond2}
		\frac{d\mu_{Y}}{d\mu_W}\left(Y\right) = \int_{C[0, T]}\frac{d\mu_{\xi,Y}}{d(\mu_\xi\times\mu_W)}\left(x,Y  \right) \mu_\xi(dx) \ \ \p\as.
	\end{equation}	
	
	\subsection{Proof of 1)}
	
	Applying Lemma \ref{lemma_relative_entropy_convex} with $U = Y^{(n)}$, $V = Y$, $Z = \xi$ and $\tilde{V} = W$, we arrive at
	\begin{align*}
		\D{\mu_{Y^{(n)}}}{\mu_Y} &\leq \int_{C[0, T]} \D{\mu_{Y^{(n)}, \xi }(\cdot,x)}{\mu_{Y,\xi} (\cdot,x)} \mu_\xi(dx) \\
		&= \int_{C[0, T]} \left(\E\left[\log \frac{d\mu_{\xi,Y^{(n)}}}{d\left(\mu_\xi \times \mu_W \right)}\left(x,Y^{(n)}\right)  \right] - \E\left[\log\frac{d\mu_{\xi,Y}}{d\left(\mu_\xi \times \mu_W \right)}\left(x,Y\right)  \right] \right)   \mu_\xi(dx) \\
		&= \int_{C[0, T]} \left(\E\left[\log \density{Y^{x,(n)}}{W}{Y^{(n)}}  \right] - \E\left[\log\density{Y^x}{W}{Y^{(n)}}  \right] \right) \mu_\xi(dx),
	\end{align*}
	where for the last equality we have used the relations (\ref{eq_relate_fixed_notfixed1}) and (\ref{eq_relate_fixed_notfixed2}).
	
	As discussed in Sections~\ref{ACRM} and~\ref{useful-relations}, $ \mu_{Y^{x,(n)}} \sim \mu_W$ for any fixed $x \in C[0, T]$ and  $\mu_{Y^{(n)}} \ll \mu_W$ for every $n \in \naturalnumber_0$. It then immediately follows that $\mu_{Y^{(n)}} \ll \mu_{Y^{x, (n)}}$, which, together with~\cite[Lemma 4.10]{liptser}, implies that
	   $$
	   \density{Y^{x,(n)}}{W}{Y^{(n)}}  = \exp\left(\int_0^T \Phi^{(n)}\left(t,Y^{(n)}\right)dY^{(n)}(t) - \frac{1}{2}\int_0^T\left( \Phi^{(n)}\left(t,Y^{(n)}\right)\right)^2  dt \right).
		$$
In a parallel fashion, we deduce that $\mu_{Y^{(n)}} \ll \mu_{Y^x}$, and moreover,
	$$
	\density{Y^x}{W}{Y^{(n)}} = \exp\left(\int_0^T f\left(t,x,Y^{(n)}\right)dY^{(n)}(t) - \frac{1}{2} \int_0^T f^2\left(t,x,Y^{(n)}\right)dt \right).
    $$	

	With $\Phi^{(n)}$ as defined in (\ref{eq_pp_conditional}), we now show that, for each $t\in[0,T]$,
	\begin{equation} \label{conditional-expectation-with-further-conditioning}
		\Phi^{(n)}\left(t,Y^{(n)}\right) = \E\left[f\left(t,x,Y^{(n-1)}\right) \big| \filtration{Y^{(n)} }{t} \right] \ \p\as.
	\end{equation}
	To see this, recall that from Section~\ref{useful-relations} that for each $n \in \naturalnumber_0$, there exists a non-anticipative functional $Q^{(n)}$ satisfying (\ref{conditional-on-x}). Then, for any bounded $\filtration{Y^{(n)}}{t}$-measurable random variable $\lambda = \lambda\left(Y^{(n)}(t)\right)$, we have
	\begin{align*}
	\E\left[ \Phi^{(n)}\left(t,Y^{(n)}\right)\lambda\left(Y^{(n)}(t)\right)\right] &= \int_{C[0, T]}\left( \int_{C[0, T]} \Phi^{(n)}\left(t,Q^{(n)}(\cdot,x,w)\right) \lambda\left(Q^{(n)}(\cdot,x,w)\right) \mu_W(dw)\right)\mu_\xi(dx)\\
	&= \int_{C[0, T]} \int_{C[0, T]} f\left(t,x,Q^{(n)}(\cdot,x,w) \right) \lambda\left(Q^{(n)}(\cdot,x,w)\right)\mu_W(dw)\times\mu_\xi(dx) \\
	&=\E\left[f\left(t,x,Y^{(n)}\right)\lambda\left(Y^{(n)}(t)\right) \right],
	\end{align*}
	where the second equality follows from (\ref{eq_pp_conditional}). The claim (\ref{conditional-expectation-with-further-conditioning}) then immediately follows.
	
	Then, using (\ref{conditional-expectation-with-further-conditioning}) and rewriting $f\left(t, x, Y^{(n)}\right)$ as $f_x^{(n)}(t)$, we have
	{\small \begin{align}
			\D{\mu_{Y^{(n)}}}{\mu_Y} & \leq \int_{C[0, T]} \E\left[\int_0^T \E\left[f_x^{(n-1)}(t) - f^{(n)}_x(t) \big|\filtration{Y^{(n)}}{t}  \right]dY^{(n)}(t) \right. \nonumber\\
			&\hspace{1.5cm} - \left. \frac{1}{2} \int_0^T\left( \E^2\left[f_x^{(n-1)}(t) \big| \filtration{Y^{(n)} }{t} \right] -  \left(f_x^{(n)}(t)\right)^2\right) dt\right] \mu_\xi(dx) \nonumber\\
			&= \int_{C[0, T]} \left( \E\left[\int_0^T \beta_x^{(n)}(t)dW(t)\right] + \frac{1}{2}\E\left[  \int_0^T \left(\beta_x^{(n)}(t)\right)^2dt\right]\right) \mu_\xi(dx), \label{last-inequality}
	\end{align}}
	where $\beta_x^{(n)}(t) = \E\left[f_x^{(n-1)}(t) - f^{(n)}_x(t) \big|\filtration{Y^{(n)}}{t}  \right]$. Now, by Jensen's inequality and the It\^o isometry, we infer that, for the first term in (\ref{last-inequality}),
	\begin{align}
		\E\left[\int_0^T \beta_x^{(n)}(t)dW(t)\right] &\leq \E^{\frac{1}{2}} \left[\int_0^T \left(\beta_x^{(n)}(t)\right)^2 dt \right] \nonumber\\
		&\leq \sqrt{\int_0^T \E\left[\left( f_x^{(n-1)}(t) - f^{(n)}_x(t)   \right)^2\right] dt} \nonumber\\
		&\leq  \sqrt{K\int_0^T\E\left[\left\|Y^{(n-1)} - Y^{(n)} \right\|_t^2 \right]dt }\nonumber\\
		&\leq (c_1KT)^{\frac{1}{2}} \sqrt{\frac{(c_2T)^{n-1}}{(n-1)!}}, \label{1-term}
	\end{align}
	where $c_1$ and $c_2$ are the constants as in \autoref{thm_existence_uniqueness}. Similarly, for the second term in (\ref{last-inequality}), we have
	\begin{equation} \label{2-term}
		\E\left[  \int_0^T \left(\beta_x^{(n)}(t)\right)^2dt\right] \leq c_1KT \frac{(c_2T)^{n-1}}{(n-1)!}.
	\end{equation}
	With the two bounds in (\ref{1-term}) and (\ref{2-term}), we eventually reach
	\begin{equation} \label{eq_relative_entropy_estimate}
		\D{\mu_{Y^{(n)}}}{\mu_Y} \leq (c_1KT)^{\frac{1}{2}} \sqrt{\frac{(c_2T)^{n-1}}{(n-1)!}} +  c_1KT \frac{(c_2T)^{n-1}}{2(n-1)!},
	\end{equation}
	which in turn implies 1) of Theorem~\ref{thm1}.
	
	\subsection{Proof of 2)}
	
	We will need the following lemma, whose proof is postponed to Appendix~\ref{proof_pp_ui}.
	\begin{lemma} \label{pp_ui}
		Assume (\textbf{L}), (\textbf{G}) and the following condition:
		
		(\text{\textbf{B}}) (Uniform Boundedness) There exists $M > 0$ such that for all $\varphi, \psi \in C[0, T]$,
		\begin{equation} \label{eq_cond_step1}
			\int_0^Tf^2(t,\phi,\varphi)dt \leq M.
		\end{equation}
		Then, for any $n\in\naturalnumber_0$ and $p$ with $|p|\geq 1$ ,
		\begin{equation} \label{bound-1}
			\E\left[ \left(\frac{d\mu_{Y^{(n)}}}{d\mu_W}\left(Y^{(n)}\right) \right)^p\right] \leq e^{\frac{p(p+1)M}{2}}.
		\end{equation}
		Moreover, for any $x \in C[0, T]$, we have
		\begin{equation} \label{bound-2}
			\E\left[ \left(\frac{d\mu_{Y^{x,(n)}}}{d\mu_W}\left( Y^{x,(n)} \right)\right)^p\right] \leq e^{\frac{p(p+1)M}{2}}.
		\end{equation}
	\end{lemma}
	
	Now, we are ready for the proof of 2), which consists of the following two steps:
	
	\paragraph{\bf Step 1.} In this step, we establish 2) with an extra condition (\textbf{B}).
	
	First of all, we note that for each $n \in\naturalnumber_0$, it can be readily verified that
	\begin{equation} \label{eq_mi_step1_split}
		\hspace{-5mm} \D{\mu_{Y^{(n)}}}{\mu_W}-\D{\mu_Y}{\mu_W} = \D{\mu_{Y^{(n)}}}{\mu_Y} + \E\left[\log\density{Y}{W}{Y^{(n)}} \right] - \E\left[\log\density{Y}{W}{Y} \right].
	\end{equation}
	Since it has been established in 1) that $\lim_{n \to \infty} \D{\mu_{Y^{(n)}}}{\mu_Y}=0$, to prove (2), it suffices to show that
	\begin{equation}
		\lim_{n\to\infty}\E\left[\log\density{Y}{W}{Y^{(n)}} \right] = \E\left[\log\density{Y}{W}{Y} \right].
	\end{equation}
	Now, applying Jensen's inequality and Fubini's theorem, we have, for any $p \geq 1$,
	\begin{equation}
		\label{eq_step1_split}
		\E\left[\left|\density{Y}{W}{Y^{(n)}} - \density{Y}{W}{Y} \right|^p \right] \leq \int_{C[0, T]} \E\left[\left| \density{Y^x}{W}{Y^{(n)}} - \density{Y^x}{W}{Y}\right|^p \right] \mu_\xi(dx).
	\end{equation}
	Notice that as discussed before, $\mu_{Y^{(n)}}\ll\mu_W$ for all $n\in\naturalnumber_0$, $\mu_Y\sim\mu_W$ and $\mu_{Y^x}\sim\mu_W$, we have $\mu_{Y^{(n)}} \ll \mu_{Y^x}\sim\mu_Y$, which, together with ~\cite[Lemma 4.10]{liptser}, implies that
	\begin{align}
		\label{above-intertext} \frac{d\mu_{Y^x}}{d\mu_W}\left(Y^{(n)}  \right)&= \exp\left( \int_0^T f\left(t,x,Y^{(n)}\right)dY^{(n)}(t)-\frac{1}{2} \int_0^T f^2\left(t,x,Y^{(n)}\right) dt \right), \ \ \p\as,
		\intertext{and moreover,}
		\label{below-intertext} \frac{d\mu_{Y^x}}{d\mu_W}\left(Y  \right)  &= \exp\left( \int_0^T f\left(t,x,Y\right) dY(t)-\frac{1}{2} \int_0^T f^2\left(t,x,Y\right) dt \right), \ \ \p\as.
	\end{align}
	It then follows from (\ref{above-intertext}), (\ref{below-intertext}), the condition (\textbf{B}) and the Cauchy-Schwarz inequality that
	{\small \begin{align} \label{eq_step2_split1_1}
			&\E\left[\left|  \log\left(\frac{d\mu_{Y^x}}{d\mu_W}\left(Y^{(n)}  \right)\right) - \log\left(\frac{d\mu_{Y^x}}{d\mu_W}\left(Y  \right) \right)  \right|^2    \right] \nonumber\\
			&\leq 2 \E\left[ \left|\int_0^T\left( f^{(n)}_x(t)f^{(n-1)}_x(t) -\frac{1}{2}\left(f^{(n)}_x(t)\right)^2 -\frac{1}{2}f^2_x(t)\right) dt\right|^2 \right] \nonumber\\
			& \hspace{2cm} +  2 \E\left[\left| \int_0^T \left( f_x^{(n)}(t)-f_x(t)  \right)dW(t)  \right|^2   \right] \nonumber \\
			& \leq 18\E\left[\left| \int_0^T f_x^{(n)}(t)\left( f_x^{(n-1)}(t)-f_x(t)  \right) dt\right|^2\right]   + 9\E\left[\left| \int_0^T f_x^{(n)}(t)\left( f_x(t)-f_x^{(n)}(t)  \right)dt\right|^2\right]\nonumber\\
			& \hspace{2cm}+9\E\left[ \int_0^T\left|f_x(t)\left(f^{(n)}_x(t)-f_x(t)  \right)dt\right|^2\right]  + 2\E\left[ \int_0^T \left( f_x^{(n)}(t)-f_x(t) \right)^2dt \right] \nonumber\\
			&\leq 18M\left( \E\left[ \int_0^T \left( f_x^{(n-1)}(t)-f_x(t)  \right)^2 dt \right] + \E\left[ \int_0^T \left( f_x(t)-f_x^{(n)}(t)  \right)^2dt \right]  \right)  \nonumber \\
			& \hspace{2cm} + 2\E\left[ \int_0^T \left( f_x^{(n)}(t)-f_x(t) \right)^2dt \right] \nonumber \\
			&\lesssim_{M,K,L} \frac{(c_2T)^{n-1}}{(n-1)!} + \frac{(c_2T)^{n}}{n!},
	\end{align}}
	where $c_2 = 2KT$, the constants $M, K, L$ in the last inequality does not depend on $x$ (cf. \autoref{thm_existence_uniqueness} and Remark \ref{remark_appendix}), and we have rewritten $f\left(t, x,Y^{(n)}\right), f(t,x,Y)$ as $f_x^{(n)}(t), f_x(t)$, respectively, for notational simplicity.
	
	Note that it can be verified that for any $p \geq 1$ and any $x, y \in \real$,
	\begin{equation}
		\label{eq_exp_inequality}
		|e^x-e^y|^p \leq |e^x-e^y|^{p-1} (e^x\vee e^y) |x-y|.
	\end{equation}
	Using this, the Cauchy-Schwarz inequality and Proposition \autoref{pp_ui}, we have the following estimate
	\begin{align}
		\label{eq_proof_lemma2_prob}
		\E\left[\left|  \frac{d\mu_{Y^{x}}}{d\mu_W}\left(Y^{(n)}  \right)-  \frac{d\mu_{Y^x}}{d\mu_W}\left(Y  \right)\right|^p \right] &
		\lesssim_{M} \sqrt{\E\left[\left|  \log\left(\frac{d\mu_{Y^x}}{d\mu_W}\left(Y^{(n)}  \right)\right) - \log\left(\frac{d\mu_{Y^x}}{d\mu_W}\left(Y  \right) \right)  \right|^2    \right]}\nonumber \\
		&\lesssim_{M,K,L,p} \sqrt{ \frac{(c_2T)^{n-1}}{(n-1)!} + \frac{(c_2T)^{n}}{n!} }.
	\end{align}
	Since the bound in (\ref{eq_proof_lemma2_prob}) does not depend on $x$, we infer from (\ref{eq_step1_split}) that
	\begin{equation}
		\label{eq_step1_split1}
		\E\left[\left|\density{Y}{W}{Y^{(n)}} - \density{Y}{W}{Y} \right|^p \right] \lesssim_{M,K,L,p} \sqrt{ \frac{(c_2T)^{n-1}}{(n-1)!} + \frac{(c_2T)^{n}}{n!}}.
	\end{equation}
	Now, using the well-known inequality $|\log x|\leq |x-1|+\left|1/x-1 \right|$, Proposition \ref{pp_ui}, H\"older's inequality and (\ref{eq_step1_split1}), we arrive at
	\begin{align} \label{eq_step1_semifinal}
		&\E\left[\left|\log\left(\frac{d\mu_Y}{d\mu_W}\left(Y^{(n)}\right)\right)-\log\left(\frac{d\mu_Y}{d\mu_W}\left(Y\right) \right)\right| \right] \nonumber \\
		&\leq\E\left[\left( \left(\frac{d\mu_Y}{d\mu_W}\left(Y\right) \right)^{-1} + \left(\frac{d\mu_Y}{d\mu_W}\left(Y^{(n)}\right) \right)^{-1}  \right)\left| \frac{d\mu_Y}{d\mu_W}\left(Y^{(n)}\right) -\frac{d\mu_Y}{d\mu_W}(Y) \right| \right] \nonumber \\
		&\lesssim_{M,p} \E^{\frac{1}{p}}\left[\left| \frac{d\mu_{Y}}{d\mu_W}\left(Y^{(n)}\right) -\frac{d\mu_Y}{d\mu_W}(Y) \right|^p  \right] \nonumber \\
		&\lesssim_{M,K,L,p}  \left(\frac{(c_2T)^n}{n!} + \frac{(c_2T)^{n-1}}{(n-1)!}\right)^{\frac{1}{2p}}.
	\end{align}		
	Therefore, by  (\ref{eq_mi_step1_split}), (\ref{eq_relative_entropy_estimate}) and (\ref{eq_step1_semifinal}), we reach
	{\small \begin{equation} \label{eq_step1_final}
			\hspace{-1cm} \left|\D{\mu_{Y^{(n)}}}{W} - \D{\mu_Y}{\mu_W}   \right| \lesssim_{M,K,L,p}  \left(\frac{(c_2T)^{n-1}}{(n-1)!} + \frac{(c_2T)^{n}}{n!}\right)^{\frac{1}{2p}} +  \sqrt{\frac{(c_2T)^{n-1}}{(n-1)!}} +  \frac{(c_2T)^{n-1}}{(n-1)!},
	\end{equation}}
	which converges to $0$ as $n$ tends to infinity, and thereby completing \textbf{Step 1}.
	
	\begin{remark}
        At the expense of a possible increase of the corresponding constant, the rate of convergence as in (\ref{eq_step1_final}) can be made arbitrarily close to $\mathcal{O}\left(\sqrt{\frac{(c_2T)^{n-1}}{(n-1)!}}\right)$ by choosing $p$ close enough to $1$.
	\end{remark}
	
	\paragraph{\bf Step 2.} In this step, we finish the proof of 2) without the condition (\textbf{B}).
	
	First of all, for each $n, m \in \naturalnumber$, we define the stopping times $\tau^{(n)}_m$ and $\tau_m$ by
	\begin{align}
		\label{eq_stopping}
		\tau_m^{(n)}&:= \begin{cases}
			\inf\{t\leq T: \int_0^t f^2(s,\xi,Y^{(n)})ds \geq m  \}, &\text{for $\int_0^Tf^2\left(t,\xi,Y^{(n)}\right)dt\geq m$;}\\
			T, &\text{otherwise,}
		\end{cases}\\
		\intertext{and}
		\tau_m &:= \begin{cases}
			\inf\{t\leq T: \int_0^t f^2(s,\xi,Y)ds \geq m  \}, &\text{for $\int_0^Tf^2(t,\xi,Y)dt\geq m$;}\\
			T, &\text{otherwise.}
		\end{cases}
	\end{align}
	Furthermore, we define the stopping time $\sigma_m$ by
	\begin{align}
		\sigma_m&:= \inf_{n \geq 0} \left\{ \tau_m^{(n)} \right\}.
	\end{align}
	For each fixed $n \in \naturalnumber_0$, it is easy to see that $\tau_m^{(n)} \uparrow T$ $\p \as$ as $m \to\infty$; and moreover, by the easily verifiable fact that
	\begin{equation}
	\label{eq_int.f_converges}
	\lim_{n\to\infty}\int_0^T f^2\left(t,\xi,Y^{(n)}\right)dt = \int_0^T f^2(t,\xi,Y)dt,
	\end{equation}
	we infer that $\sigma_m \uparrow T$ $\p \as$ as $m \to\infty$. Now, we are ready to define truncated versions of $f$ and $Y$ as below:
	\begin{equation}
		f_{(m)}(t,\phi,\varphi):=f(t,\phi,\varphi)\1_{\{\int_0^t f^2(s,\phi,\varphi)ds \leq m  \}},
	\end{equation}	
	and
	\begin{equation}
		\label{eq_truncate_1}
		Y_{(m)}(t):= \int_0^t f_{(m)}(s,\xi,Y)ds + W(t), \ 0\leq t\leq T.
	\end{equation}
	Since for any $t \in[0,T]$,  $Y_{(m)}(t)\1_{\{t\leq \tau_m\}}=Y(t)\1_{\{t\leq \tau_m\}} \ \p\as$, and both the processes  $\{Y_{(m)}(t)\1_{\{t\leq \tau_m\}}; 0\leq t\leq T  \}$ and $\{Y(t)\1_{\{t\leq \tau_m\}} \}$ admit left-continuous sample paths, they are indistinguishable to each other. Consequently, we have
	\begin{equation}
		\p\left( \int_0^t f_{(m)}(s,\xi,Y)ds = \int_0^t f_{(m)}\left(s,\xi,Y_{(m)} \right)ds;\ 0\leq t\leq T \right)=1,
	\end{equation}
	and moreover,
	\begin{equation}
		Y_{(m)}(t) = \int_0^t f_{(m)}(s,\xi,Y_{(m)})ds+W(t), \ 0\leq t\leq T.
	\end{equation}
	Note that, for any fixed $m \in \naturalnumber$, the corresponding Picard's iteration $\left\{Y^{(n)}_{(m)}\right\}_{n=0}^\infty$ recursively computes: for any $n \in \naturalnumber$,
	\begin{equation} \label{eq_picard_truncate}	
		Y_{(m)}^{(n)}(t)= \int_0^t f_{(m)}\left(s,\xi,Y_{(m)}^{(n-1)}\right)ds + W(t),
	\end{equation}
	starting with $Y_{(m)}^{(0)}(t) \equiv 0$. By way of induction, we can show that, for all $n, m \in \naturalnumber_0$ and $t\in[0,T]$,
	\begin{equation} \label{eq_truncate_nontruncate}
		\hspace{-0.5cm} Y_{(m)}^{(n)}(t)\1_{\{t\leq \sigma_m\}} = Y^{(n)}(t)\1_{\{t\leq \sigma_m\}} \ \p \as \quad\text{and}\quad Y_{(m)}(t)\1_{\{t\leq \sigma_m\}}=Y(t)\1_{\{t\leq \sigma_m\}}\ \p \as,
	\end{equation}
	i.e., the process $\{ Y^{(n)}_{(m)}(t) \1_{\{t\leq \sigma_m\}}; 0\leq t\leq T\}$ is a modification of $ \left\{Y^{(n)}(t) \1_{\{t\leq \sigma_m\}};0\leq t\leq T\right\}$. Again, since both of them admit left-continuous sample paths, they are indistinguishable to each other. And a similar argument can be applied to establish the indistinguishability between the two processes $\left\{ Y_{(m)}(t) \1_{\{t\leq \sigma_m\}}; 0\leq t\leq T \right\}$ and $\left\{ Y(t) \1_{\{t\leq \sigma_m\}}; 0\leq t\leq T \right\}$. Hence, when restricted to the event $\{\sigma_m=T\}$, the processes $Y^{(n)}_{(m)}$ and $Y_{(m)}$ share the same laws as $Y^{(n)}$ and $Y$, respectively.
	
	With the above preparations, we are now ready to establish the desired convergence in 2). Splitting $\E\left[\log\density{Y}{W}{Y} \right]$ with respect to the two non-overlapping events $\{\sigma_m = T\}$ and $\{\sigma_m <T\}$, we have
	\begin{align}
		\label{eq_step2_split_y}
		\E\left[\log\density{Y}{W}{Y} \right] &= \E\left[\log\density{Y}{W}{Y} \left(\1_{\{\sigma_m=T \}} + \1_{\{\sigma_m<T \}}  \right) \right] \nonumber \\
		&=\E\left[\log\density{Y_{(m)}}{W}{Y_{(m)}} \right] + \E\left[\log \frac{d\mu_Y}{d\mu_W}(Y) \1_{\{\sigma_m<T \} } \right] .
	\end{align}
	Similarly, for each $n,m\in\naturalnumber$, by (\ref{eq_truncate_nontruncate}), we have
	{\small \begin{equation} \label{eq_step2_split_yn}
			\hspace{-1cm} \E\left[\log\left(\frac{d\mu_{Y^{(n)}}}{d\mu_W}\left(Y^{(n)}\right)\right) \right] = \E\left[\log\left(\frac{d\mu_{{Y_{(m)}^{(n)}}}}{d\mu_W}\left(Y^{(n)}_{(m)}\right)\right)  \right]  + \E\left[\log\left(\frac{d\mu_{Y^{(n)}}}{d\mu_W}\left(Y^{(n)}\right)\right)\1_{\{ \sigma_m<T \}} \right].
	\end{equation}}
	Note that by \textbf{Step 1}, we have
	\begin{equation} \label{Step1-Result}
		\lim_{n \to \infty} \E\left[\log\left(\frac{d\mu_{{Y_{(m)}^{(n)}}}}{d\mu_W}\left(Y^{(n+1)}_{(m)}\right)\right)  \right]=\E\left[\log\density{Y_{(m)}}{W}{Y_{(m)}} \right].
	\end{equation}
	And, by (\ref{eq_step2_split_y}), (\ref{eq_step2_split_yn}) and (\ref{Step1-Result}), we have
	\begin{align} \label{eq_step2_split}
		& \ \ \ \E\left[\log\left(\frac{d\mu_Y}{d\mu_W}(Y)\right) \right] - \lim_{n\to\infty}	\E\left[\log\left(\frac{d\mu_{Y^{(n)}}}{d\mu_W}\left(Y^{(n)}\right)\right) \right]\nonumber\\& = \E\left[\log\left(\frac{d\mu_Y}{d\mu_W}(Y)\right) \1_{\{\sigma_m<T \} } \right]  - \lim_{n\to\infty} \E\left[\log\left(\frac{d\mu_{Y^{(n)}}}{d\mu_W}\left(Y^{(n)}\right)\right)\1_{\{ \sigma_m<T \}} \right].
	\end{align}	
	Noticing that (\ref{eq_step2_split}) holds true for all $m \geq 0$ and using the fact that $\1_{\{\sigma_m<T \}}\downarrow 0$ as $m\to\infty$, we deduce that
	\begin{equation}
		\lim_{m\to\infty}\E\left[\log\left(\frac{d\mu_Y}{d\mu_W}(Y)\right) \1_{\{\sigma_m<T \} } \right] = 0,
	\end{equation}
	where we have applied the dominated convergence theorem. The proof is then complete if we can show
	\begin{equation}
		\label{eq_step2_split2}
		\lim_{m\to\infty}\lim_{n\to\infty} \E\left[\log\left(\frac{d\mu_{Y^{(n)}}}{d\mu_W}\left(Y^{(n)}\right)\right)\1_{\{ \sigma_m<T \}} \right]=0.
	\end{equation}
	To this end, we note that by (\ref{eq_logderivative_uncondition}),
	\begin{multline}
		\E\left[\log\left(\frac{d\mu_{Y^{(n)}}}{d\mu_W}\left(Y^{(n)}\right)\right)\1_{\{ \sigma_m<T \}}\right] =\\ \E\left[\left(\int_0^T \alpha^{(n)}(t)dW(t)    - \frac{1}{2} \int_0^T \left(\alpha^{(n)}(t)\right)^2dt   +  \int_0^T f^{(n-1)}(t) \alpha^{(n)}(t)  dt\right)\1_{\{ \sigma_m<T \}}\right],
	\end{multline}
	where $\alpha^{(n)}(t) := \E\left[ f\left(t,\xi,Y^{(n-1)}\right) \big| \filtration{Y^{(n)}}{t}\right]$ and $f^{(n)}(t)$ represents $f(t, \xi, Y^{(n)})$. By the It\^o isometry, the Cauchy-Schwarz inequality and Jensen's inequality, we have
	\begin{align}
		\E\left[\left|\int_0^T \alpha^{(n)}(t)dW(t) \1_{\{\sigma_m<T\} }\right| \right] &\leq\sqrt{ \E\left[ \int_0^T \left(\alpha^{(n)}(t)\right)^2dt \right] \p(\sigma_m<T) } \nonumber\\
		&\leq \sqrt{ \E\left[ \int_0^T \left(f^{(n-1)}(t)\right)^2dt \right] \p(\sigma_m<T)}.
    \end{align}
    By (\ref{eq_int.f_converges}), with $f(t)$ represents $f(t, \xi, Y)$, we have
    \begin{equation}
		\lim_{n\to\infty}\E\left[\left|\int_0^T \alpha^{(n)}(t)dW(t) \1_{\{\sigma_m<T\} }\right| \right] \leq \sqrt{ \E\left[ \int_0^T  f^2(t)dt \right] \p(\sigma_m<T) } ,
	\end{equation}
	and therefore,
	\begin{equation}
		\label{eq_proof_lemma4.3_key_1}
		\lim_{m\to\infty}\lim_{n\to\infty}	\E\left[\int_0^T \alpha^{(n)}(t)dW(t) \1_{\{\sigma_m<T\} } \right]=0.
	\end{equation}
	Using a similar argument, we deduce that
	\begin{align}
		\hspace{-1cm} \E\left[\left|\int_0^T f^{(n-1)}(t) \alpha^{(n)}(t)  dt\1_{\{ \sigma_m<T \}}\right| \right] & \leq \E\left[\left(\int_0^T \left(f^{(n-1)}(t)\right)^2dt\right)^{1/2} \left(\int_{0}^{T}\left(\alpha^{(n)}(t) \right)^2dt\right)^{1/2} \1_{\{ \sigma_m<T \}}   \right] \nonumber \\
		&\hspace{-2cm} \leq\sqrt{ \E\left[\int_0^T \left(f^{(n-1)}(t)\right)^2dt \1_{\{ \sigma_m<T \}}   \right]\E\left[ \int_{0}^{T}\left(\alpha^{(n)}(t) \right)^2dt\right]} \nonumber\\
		&\hspace{-2cm} \leq \sqrt{ \E\left[\int_0^T \left(f^{(n-1)}(t)\right)^2dt \1_{\{ \sigma_m<T \}}   \right]\E\left[ \int_{0}^{T}\left(f^{(n-1)}(t) \right)^2dt\right]},
	\end{align}
	which, upon letting $n\to\infty$, yields
	\begin{equation}
		\hspace{-1cm} \lim_{n\to\infty}\E\left[\left|\int_0^T f^{(n-1)}(t) \alpha^{(n)}(t)  dt\1_{\{ \sigma_m<T \}}\right| \right] \leq \sqrt{ \E\left[\int_0^T f^2(t)dt \1_{\{ \sigma_m<T \}} \right] \E\left[ \int_{0}^{T}f^2(t)dt\right]}.
	\end{equation}
	A further application of the dominated convergence theorem then gives
	\begin{equation} \label{eq_proof_lemma4.3_key_2}
		\lim_{m\to\infty}\lim_{n\to\infty}\E\left[\int_0^T f^{(n-1)}(t) \alpha^{(n)}(t)  dt\1_{\{ \sigma_m<T \}} \right]=0.
	\end{equation}
	Next, by Jensen's inequality and (\textbf{L}), we have
	{\small \begin{align} \label{eq_proof_lemma4.3_key_3_1}
			&\E\left[\int_0^T \left(\alpha^{(n)}(t)\right)^2 dt \1_{\{\sigma_m<T \}} \right] \nonumber\\
				&\leq \E\left[\int_0^T \E\left[\left(f^{(n-1)}(t)\right)^2 \Big|\filtration{Y^{(n)}}{t}   \right]dt \1_{\{\sigma_m<T \}}   \right] \nonumber \\
			&\leq \E\left[ \int_0^T  \E\left[ L\left(1+ \sup_{0 \leq s \leq t}|\xi(s) |^2 + \sup_{0 \leq s \leq t}\left|Y^{(n-1)}(s) \right|^2 \right)  \Big| \filtration{Y^{(n)}}{t}  \right]dt   \1_{\{\sigma_m<T \}} \right] \nonumber\\
			&\leq \E\left[ \int_0^T  \E\left[ L\left(1+\sup_{0 \leq s \leq t}|\xi(s) |^2 + 2\sup_{0 \leq s \leq t}\left|Y^{(n)}(s)- Y^{(n-1)}(s) \right|^2 \right)  \Big| \filtration{Y^{(n)}}{t}   \right]dt   \1_{\{\sigma_m<T \}} \right] \nonumber\\
			&\hspace{2cm} + 2\E\left[ \int_0^T \sup_{0 \leq s \leq t}\left|Y^{(n)}(s) \right|^2 dt \1_{\{\sigma_m<T \}} \right]  \\
			&\leq L\left( T\p(\sigma_m<T) + \E\left[\int_0^T \E\left[\|\xi\|_T^2 + 2\left\|Y^{(n)}- Y^{(n-1)} \right\|_T^2 \Big| \filtration{Y^{(n)}}{t} \right]dt \1_{\{\sigma_m<T\}} \right] \right. \nonumber\\
			&\hspace{2cm} + \left. 2T\E\left[\left\| Y^{(n)} \right\|^2_T \1_{\{\sigma_m<T \}} \right] \right).
			\end{align}}
	It is easy to see that
	\begin{equation}
		\label{eq_proof_lemma4.3_key_3_2}
		\lim_{m\to\infty}\lim_{n\to\infty} \E\left[\left\| Y^{(n)} \right\|^2_T \1_{\{\sigma_m<T \}} \right] = \lim_{m\to\infty}\E\left[\left\| Y \right\|^2_T \1_{\{\sigma_m<T \}} \right]=0,
	\end{equation}
	and
	\begin{equation}
		\label{eq_proof_lemma4.3_key_3_3}
		\lim_{n\to\infty} \E\left[\int_0^T \E\left[ \left\|Y^{(n)}- Y^{(n-1)} \right\|_T^2 \bigg| \filtration{Y^{(n)}}{t} \right]dt  \1_{\{\sigma_m<T\}} \right] \leq  T\lim_{n\to\infty} \E\left[ \left\|Y^{(n)}- Y^{(n-1)} \right\|_T^2\right] =0.
	\end{equation}
	Finally, since $\E\left[\|\xi\|_T^2\right] <\infty$, by \cite[Theorem 5.5.1]{PTE}, the collection
	$$
	\left\{\E\left[\|\xi\|_T^2 \big| \filtration{Y^{(n)}}{t} \right]  ; \ n\in\naturalnumber_0, \ 0\leq t\leq T\right\}
	$$
	is uniformly integrable. It then follows that
	\begin{equation}
		\label{eq_proof_lemma4.3_key_3_4}
		\lim_{m\to\infty}\limsup_{n\to\infty} \E\left[\int_0^T \E\left[\|\xi \|_T^2 \Big| \filtration{Y^{(n)}}{t} \right]\1_{\{\sigma_m<T \} }dt   \right]=0.
	\end{equation}
	Combining (\ref{eq_proof_lemma4.3_key_3_1}), (\ref{eq_proof_lemma4.3_key_3_2}), (\ref{eq_proof_lemma4.3_key_3_3}) and (\ref{eq_proof_lemma4.3_key_3_4}), we infer that
	\begin{equation}
		\label{eq_proof_lemma4.3_key_3}
		\lim_{m\to\infty}\lim_{n\to\infty} \E\left[\int_0^T \left(\alpha^{(n)}(t)\right)^2 dt \1_{\{\sigma_m<T \}} \right]=0.
	\end{equation}
	Hence, the desired (\ref{eq_step2_split2}) follows from (\ref{eq_proof_lemma4.3_key_1}), (\ref{eq_proof_lemma4.3_key_2}) and (\ref{eq_proof_lemma4.3_key_3}), which in turn completes the proof of the theorem.
	
	\section{Application in Information Theory} \label{sec_proof}

	In the context of information theory, an SDE taking the form
	\begin{equation} \label{eq_gaussian_channel}
		Y(t) = \int_0^t f(s,\xi,Y)ds + W(t), \ 0\leq t \leq T,
	\end{equation}
	is often used to model a continuous-time Gaussian channel with feedback, where $\xi$ is interpreted as the message to be transmitted, $f$ as the channel input and $Y$ as the channel output. The {\em mutual information} $I_T(\xi;Y)$, a fundamental notion measuring the information transmission rate of the channel (\ref{eq_gaussian_channel}), is defined as
	\begin{equation} \label{convergence-rate}
		I_T(\xi;Y) := \E\left[\log \frac{d\mu_{\xi,Y}}{d\left(\mu_\xi  \times \mu_Y \right) }(\xi,Y) \right]= \D{\mu_{\xi,Y}}{\mu_\xi \times \mu_Y }.
	\end{equation}
	
	Exemplifying the possible applications of Theorem~\ref{thm1}, the following theorem show that $I_T(\xi;Y)$ can be approximated via the means of Picard's iteration.
	\begin{theorem} \label{pp_mi}
		Let $\{Y^{(n)}\}_{n=0}^\infty$ be as in (\ref{eq_picard}) with $g \equiv 1$ and $\xi=\{\xi(t),\mathcal{F}_t;0\leq t\leq T \}$ be such that $\E\left[\|\xi\|_T^2 \right]<\infty$. Suppose $f$ and $g$ satisfy (\textbf{L}) and (\textbf{G}). Then,
		\begin{equation} \label{eq_result_mi}
			\lim_{n\to\infty}I_T\left(\xi;Y^{(n)} \right) = I_T(\xi;Y).
		\end{equation}
		In addition, if $f$ further satisfies (\textbf{B}), then, for any $p \geq 1$, we have
		\begin{equation}
			\label{eq_result_mi2}
			\left|I_T\left(\xi;Y^{(n)}\right) - I_T(\xi;Y) \right| \lesssim_{M,K,L,p}  \left(\frac{(cT)^{n-1}}{(n-1)!} + \frac{(cT)^{n}}{n!}\right)^{\frac{1}{2p}} +  \sqrt{\frac{(c_2T)^{n-1}}{(n-1)!}} +  \frac{(c_2T)^{n-1}}{(n-1)!},
		\end{equation}
		where $c_1,c_2$ are the constants as in \autoref{thm_existence_uniqueness}.
	\end{theorem}

	\begin{proof}
		
		First of all, we note that the mutual information $I_T(\xi;Y)$ can be split into
		\begin{align}
			I_T(\xi;Y)  &= \E \left[\log\left( \frac{d\mu_{\xi,Y}}{d(\mu_\xi\times \mu_{Y})}\left(\xi ,Y \right)\right)  \right] \nonumber \\
			&= \E\left[\log\left(\frac{d\mu_{\xi,Y}}{d(\mu_\xi\times\mu_W)}\left(\xi,Y\right) \right) \right] - \E\left[\log \left(\frac{d\mu_{Y}}{d\mu_W} \left(Y\right) \right)\right] \nonumber \\
			&=\D{\mu_{\xi,Y}}{\mu_\xi \times \mu_W } - \D{\mu_{Y}}{\mu_W}.
		\end{align}
		Similarly, for each $n\in\naturalnumber$, we have
		\begin{equation}
			\label{eq_split}
			I_T\left(\xi;Y^{(n)}\right) =\D{\mu_{\xi,Y^{(n)}}}{\mu_\xi \times \mu_W } - \D{\mu_{Y^{(n)}}}{\mu_W}.
		\end{equation}
		Then, by 2) of~\autoref{thm1}, we have
		\begin{equation}
			\lim_{n\to\infty} \D{\mu_{Y^{(n)}}}{\mu_W} = \D{\mu_Y}{\mu_W}.
		\end{equation}
		Hence, to prove (\ref{eq_result_mi}), it remains to show
		\begin{equation}
			\label{eq_step_mi}
			\lim_{n\to\infty} \D{\mu_{\xi,Y^{(n)}}}{\mu_W} = \D{\mu_{\xi,Y}}{\mu_W}.
		\end{equation}
		Towards this goal, applying Proposition \ref{pp_girs2}, we have
		\begin{align}
			\frac{d\mu_{\xi,Y^{(n)}}}{d\left(\mu_\xi\times\mu_W \right)} \left(\xi,Y^{(n)}\right) &= \exp\left( \int_0^T \beta^{(n)}(t)dY^{(n)}(t)-\frac{1}{2} \int_0^T \left(\beta^{(n)}(t) \right)^2 dt   \right) \nonumber \\
			&\hspace{-3cm} =\exp\left( \int_0^T \beta^{(n)}(t)dW(t) + \int_0^T \beta^{(n)}(t)f^{(n-1)}(t) dt -\frac{1}{2}\int_0^T\left(\beta^{(n)}(t) \right)^2 dt   \right),
		\end{align}
		and furthermore,
		\begin{equation}
			\frac{d\mu_{\xi,Y}}{d\left(\mu_\xi\times \mu_W \right)}(\xi,Y) = \exp\left(\int_0^T f(t,\xi,Y)dY(t) - \frac{1}{2} \int_0^T f^2(t,\xi,Y) dt  \right).
		\end{equation}
		Here,
$
\beta^{(n)}(t) = \E\left[f^{(n-1)}(t)  \big| \filtration{\xi,Y^{(n)}}{t}  \right].$
Now, with a parallel argument as in the proof of Theorem~\ref{thm1}, we reach
		\begin{multline}
			\label{eq_mi_step1_final}
			\E\left[\left| \log\left(\frac{d\mu_{\xi,Y^{(n)}}}{d\left(\mu_\xi\times\mu_W \right)} \left(\xi,Y^{(n)}\right)  \right)   - \log \left(\frac{d\mu_{\xi,Y}}{d(\mu_\xi\times\mu_W)}\left(\xi, Y\right) \right)  \right|    \right] \\ \leq 4 (c_1KT)^{1/2} \sqrt{\frac{(c_2T)^{n}}{n!}} + 4 (c_1KT)^{1/2} \sqrt{\frac{(c_2T)^{n-1}}{(n-1)!}} + 2c_1KT \frac{(c_2T)^{n}}{n!},
		\end{multline}
		which implies (\ref{eq_step_mi}) and hence (\ref{eq_result_mi}). In addition, if $f$ satisfies the condition (\textbf{B}), we can combine (\ref{eq_mi_step1_final}) and (\ref{eq_relative_entropy_estimate}) to establish (\ref{eq_result_mi2}).
	\end{proof}
	
\begin{remark}
Effective computation of the mutual information of a continuous-time Gaussian feedback channel under various input constraints has been a long open problem in information theory (see, e.g.,~\cite{Shannon, Ducan, ka71, Snyders, zakai}). Evidently, Theorem~\ref{pp_mi} suggests a solution via Picard's iteration, which, thanks to the rate of convergence in (\ref{convergence-rate}), can be particularly promising for the channel under the condition (\textbf{B}), which is often termed as ``peak power constraint'' (see, e.g.,~\cite{ozarow}) in information theory.
\end{remark}

	\section{Appendices} \appendix
	
	\section{Proof of Theorem~\ref{thm_existence_uniqueness}} \label{proof_thm_existence_uniqueness}
	
	The proof closely follows those showing the existence and uniqueness of the solution to some ``conventional'' SDEs where $\xi$ is absent (see, e.g.,~\cite{BMSC, Mao}). In the following we only sketch the proof and emphasize the technical issues only present in our setting.
	
	The uniqueness of the solution can be proven using Gronwall's inequality in a parallel fashion as in~\cite{BMSC,Mao}, so here we only prove the existence. More precisely, we achieve this by establishing (\ref{eq_square_bound_sq}), which in turn implies the well-definedness of the limit of the sequence $\left\{Y^{(n)}\right\}$.
	
	To this end, we first note that for any $n \in \naturalnumber_0$ and any $t \in [0,T]$,
	\begin{align}
		\sup_{0\leq s\leq t}\left|Y^{(n)}(s)\right|^2 &  \leq  2\sup_{0\leq s\leq t}\left(\int_0^s f\left(u,\xi,Y^{(n-1)}\right)du \right)^2 + 2\sup_{0\leq s\leq t}\left( \int_0^s g\left(u,Y^{(n-1)}\right)dW(u) \right)^2 \nonumber \\
		&\hspace{-1cm} \leq  2t \int_0^t f^2\left(u,\xi,Y^{(n-1)}\right)du + 2\sup_{0\leq s\leq t}\left( \int_0^s g\left(u,Y^{(n-1)}\right)dW(u) \right)^2, \label{Yns}
	\end{align}
	where the last inequality follows from the Cauchy-Schwarz inequality. Taking expectation on both sides of (\ref{Yns}) then yields
	\begin{align}
		\E\left[\sup_{0 \leq s \leq t}\left|Y^{(n)}(s) \right|^2 \right] & \leq 2LT\left(T+T\E\left[\|\xi\|_T^2\right]+ \int_0^t \E\left[\sup_{0 \leq v \leq u}\left|Y^{(n-1)}(v)\right|^2 \right]du\right) \nonumber \\
                            & \quad \quad+ 2 \E\left[\sup_{0\leq s\leq t}\left( \int_0^s g\left(u,Y^{(n-1)}\right)dW(u) \right)^2\right]. \label{EYns}
	\end{align}

	For the second term in (\ref{EYns}), by Doob's submartingale inequality and the It\^o isometry, we have
	\begin{align}
		\E\left[\sup_{0\leq s\leq t} \left|  \int_0^s g\left(u,\xi,Y^{(n-1)}\right)dW(u)\right|^2 \right] &\leq 2^2  \E\left[  \int_0^t g^2\left(u,\xi,Y^{(n-1)}\right)du  \right] \nonumber \\
		&\hspace{-1cm} \leq 4 L \left( T+T\E\left[\|\xi\|_T^2\right] + \int_0^t \E\left[\sup_{0 \leq v \leq u}\left|Y^{(n-1)}(u)\right| \right]du \right).
	\end{align}
	Now, letting $k_1 := 2LT(T+4)\left(1+\E\left[\|\xi\|_T^2\right] \right) $ and $k_2 := 2L(T+4)$, we have
	$$
	\E\left[\sup_{0 \leq s \leq t}\left|Y^{(n)}(s) \right|^2 \right] \leq k_1 + k_2 \int_0^t \E\left[\sup_{0 \leq v \leq u}\left|Y^{(n-1)}(v)\right|^2 \right]du,
	$$
	a recursive application of which yields
	\begin{equation}
		\E\left[\sup_{0 \leq s \leq t}\left|Y^{(n)}(s) \right|^2 \right] \leq k_1e^{k_2t}.
	\end{equation}
	
	Next, using a similar approach as above and taking into account the fact that $f$ and $g$ satisfy (\textbf{L}), we reach, for any $t \in [0,T]$,
	\begin{align} \label{eq_existence_uniqueness_c_split}
		\hspace{-1cm} \E\left[\sup_{0 \leq s \leq t}\left|Y^{(1)}(s) - Y^{(0)}(s) \right|^2\right] & \leq 2t \int_0^t \E\left[f^2\left(u,\xi,Y^{(0)}\right) \right]du + 2 \E\left[\sup_{0\leq s\leq t} \left|  \int_0^s g\left(u,Y^{(0)}\right)dW(u)\right|^2 \right] \nonumber \\
		&\hspace{-2cm} \leq 2T^2L \left(1+\E\left[\|\xi\|_T^2 \right] \right) + 2\E\left[\sup_{0\leq s\leq t} \left|  \int_0^s g\left(u,Y^{(0)}\right)dW(u)\right|^2 \right].
	\end{align}
	Again, by Doob's submartingale inequality and the It\^o isometry,
	\begin{align} \label{missing}
		\E\left[\sup_{0\leq s\leq t} \left|  \int_0^s g\left(u,Y^{(0)}\right)dW(u)\right|^2 \right] &\leq 2^2  \E\left[  \int_0^t g^2\left(u,Y^{(0)}\right)du  \right] \nonumber \\
		&\leq 4 LT \left( 1+\E\left[\|\xi\|_T^2\right]  \right).
	\end{align}
	Hence, from (\ref{eq_existence_uniqueness_c_split}) and (\ref{missing}), we have
	\begin{equation}
		\label{eq_existence_uniqueness_c}
		\E\left[\sup_{0 \leq s \leq t} \left| Y^{(1)}(s)-Y^{(0)}(s)\right|^2 \right] \leq 2TL(T+4) \left( 1+\E\left[\|\xi\|_T^2\right]\right) =:c_1,
	\end{equation}
	for any $t\in[0,T]$. It then follows from the fact that $f$ and $g$ satisfy (\textbf{L}) that
	\begin{align*}
		\E\left[\sup_{0 \leq s \leq t} \left| Y^{(2)}(s)-Y^{(1)}(s)\right|^2 \right] & \leq 2\E\left[\left|\int_0^t \left(f\left(u,\xi,Y^{(1)}\right) - f\left(u,\xi,Y^{(0)}\right) \right)du \right|^2\right] \\
		& \hspace{0.5cm} + 8 \E\left[ \left| \int_0^t \left(g\left(u,\xi,Y^{(1)}\right) - g\left(u,\xi,Y^{(0)}\right) \right) dW(u)  \right|^2 \right]\\
		&\leq 2K(T+4)\int_0^t \E\left[\sup_{0\leq v\leq u}\left| Y^{(1)}(v)-Y^{(0)}(v) \right|^2  \right]du \\
		&\leq c_1c_2t,
	\end{align*}
	where $c_2 := 2K(T+4)$. Now, by way of induction, one can readily show that
	\begin{equation}
		\label{eq_exist_unique_cauchy}
		\E\left[\sup_{0 \leq s \leq t}\left| Y^{(n+1)}(t)-Y^{(n)}(t)\right|^2 \right] \leq c_1 \frac{(c_2t)^n}{n!},
	\end{equation}
	which can be used to establish
	$$
	\E\left[\left\|Y^{(n)} \right\|_T^2\right] \leq k_1e^{k_2T},
	$$
	and thereby completing the proof of (\ref{eq_exist_unique_bound_picard}). Consequently, the sequence $\{Y^{(n)}\}_{n=0}^\infty$ is Cauchy in $L^2(\Omega,\mathcal{F},\p)$ with the limit $Y=\{Y(t),\mathcal{F};0\leq t\leq T\}$. Furthermore, by Fatou's lemma, the limit $Y$ satisfies
	\begin{equation}
		\E\left[\|Y\|^2_T \right] \leq \liminf_{n\to\infty} \E\left[\left\| Y^{(n)}\right\|_T^2 \right] \leq k_1e^{k_2T}.
	\end{equation}
	We are now ready to establish (\ref{eq_square_bound_sq}). To this end, note that for any $t\in[0,T]$, we have
	$$
	\E\left[\sup_{0 \leq s \leq t}\left| Y^{(n)}(s)-Y(s)\right|^2 \right] \leq c_2 \int_0^t \E\left[\sup_{0 \leq v \leq u}\left|Y^{(n-1)}(v)-Y(v)\right|^2 \right]du,
	$$
	which, upon a sequence of recursive substitutions, yields that
	\begin{equation} \label{eq_exist_unique_nandlimit}
		\E\left[\sup_{0 \leq s \leq t}\left| Y^{(n)}(s)-Y(s)\right|^2 \right] \leq \E\left[\|Y\|^2_T \right] \frac{(c_2 t)^n}{n!},
	\end{equation}
	establishing the $L^2$ convergence of $\left\{Y^{(n)} \right\}_{n=0}^\infty$. To establish that $Y$ is the uniform pointwise limit of the sequence $\left\{Y^{(n)} \right\}_{n=0}^\infty$ to complete the proof of (\ref{eq_square_bound_sq}), one can use the estimate in (\ref{eq_exist_unique_cauchy}) and the Borel-Cantelli lemma; moreover, together with the estimate in (\ref{eq_exist_unique_nandlimit}), one can show that $Y$ indeed satisfies (\ref{eq_sde}). These assertions can be shown in the same way as in \cite{BMSC,Mao}, and thereby we omit the details thereof.
	
	\begin{remark}
		\label{remark_appendix}
		For the special case that $g \equiv 1$ and $f$ satisfies (\textbf{B}), $c_1$ can be replaced by $\tilde{c}_1:= 2T(M+4)$, a constant that is independent of $\xi$, which can be justified by noting
		\begin{equation}
			\E\left[\sup_{0\leq s\leq t}\left|Y^{(1)}(s) - Y^{(0)}(s) \right|^2\right] \leq 2t\E\left[\int_0^t f^2(u,\xi, Y^{(0)})du \right] + 8T \leq 2T(M+4) = \tilde{c}_1.
		\end{equation}
		It can be verified that, via a similar argument, the constants $k_1$ and $c_3$ can also be respectively replaced by some constants that are independent of $\xi$.
	\end{remark}
	
	\section{Proof of Proposition~\ref{lemma_novikov}} \label{proof_lemma_novikov}
	
	By~\cite[Theorem 7.1]{liptser}, it suffices to establish
	\begin{equation} \label{eq_novikov}
		\E\left[\exp\left(-\int_0^T f^{(n)}(t)dW(t) - \frac{1}{2}\int_0^T \left(f^{(n)}(t) \right)^2 dt\right)  \right]=1, \ \text{for all } n\in\naturalnumber_0.
	\end{equation}		
	In the meantime, to prove (\ref{eq_novikov}), by the well-known Novikov's condition~\cite{BMSC}, it suffices to construct, for each $n\in\naturalnumber$, a sequence $\{t^{(n)}_m\}_{m=0}^{M_n}$ with $0=t^{(n)}_0<t^{(n)}_1<\dots<t^{(n)}_{M_n}=T$, such that for any $m = 1, 2, \dots, M_n$,
	\begin{equation}
		\E\left[ \exp\left(\int_{t^{(n)}_{m-1}}^{t^{(n)}_m} \left(f^{(n)}(t)\right)^2dt  \right)   \right] <\infty.
	\end{equation}				
	To this end, we first note that it follows from (\textbf{L}) that, for each $n\in\naturalnumber_0$,
	\begin{align*}
		\left\|Y^{(n+1)}\right\|_T^2 &\leq 2\|W\|_T^2 + 2\int_0^T L\left(1+ \sup_{0 \leq s \leq t}|\xi(s) |^2+\sup_{0 \leq s \leq t}|Y^{(n)}(s)|^2 \right)dt\\
		&\leq 2\|W\|_T^2 + 2LT(1+\|\xi \|_T^2 ) + 2 L\int_0^T\sup_{0 \leq s \leq t}|Y^{(n)}(s)|^2dt,
	\end{align*}
	a recursive application of which yields
	\begin{align*}
		\left\|Y^{(n+1)}\right\|_T^2&\leq 2C_{n+1}\|W\|_T^2 + 2LTC_{n+1}(1+\|\xi \|_T^2),
	\end{align*}
	where $C_{n+1} =\frac{(2L)^{n+1}-1}{2L-1}$. 	
	Then, with the linear growth assumption (\textbf{G}), we derive
	\begin{align*}
		\E\left[ \exp\left(\frac{1}{2}\int_{t^{(n)}_{m-1}}^{t^{(n)}_m} \left(f^{(n)}(t)\right)^2dt  \right)  \right] &\leq 	\E\left[ \exp\left(\frac{1}{2} L\left(t^{(n)}_{m}-t^{(n)}_{m-1} \right)\left(1+\|\xi \|_T^2+  \| Y^{(n)} \|_T^2\right)    \right)      \right]\\
		&\leq e^{LTD_{n,m}}\E\left[e^{ LTD_{n,m} \|\xi\|_T^2}  \right] \E\left[e^{D_{n,m}\|W\|_T^2}   \right],
	\end{align*}
	where $D_{n,m}=LC_n\left(t^{(n)}_{m}-t^{(n)}_{m-1}\right) $. Now, applying Doob's submartingale inequality, we conclude that, as long as $\max_{m\in\{1,2,\dots,M_n \}}\{t^{(n)}_m-t^{(n)}_{m-1} \} \leq \frac{\varepsilon}{L^2C_nT}$,
	\begin{equation*}	
		\E\left[ \exp\left(\frac{1}{2}\int_{t^{(n)}_{m-1}}^{t^{(n)}_m} \left(f^{(n)}(t)\right)^2dt  \right)  \right]\leq 	4e^{LTD_{n,m}}\E\left[e^{ LTD_{n,m} \|\xi\|_T^2}  \right]\E\left[e^{D_{n,m}W(T)^2}  \right]<\infty,
	\end{equation*}
	which immediately implies the proposition.
	
	\section{Proof of Proposition \ref{pp_girs2}} \label{proof_pp_girs2}
	
	Fix $x \in C[0, T]$ and consider the process $Z^x$ satisfying
	$$
	Z^x(t) = \int_0^th(s,x,W)ds + W(t) ,\ 0\leq t\leq T.
	$$
	Using similar arguments as in Section~\ref{useful-relations}, we infer that there exists a non-anticipative functional $Q:[0,T]\times C[0,T]\to\real$, such that for any $t \in[0,T]$,
	$$
	Z(t) = Q(t,\xi,W) \mbox{ and } Z^x(t) = Q(t,x,W).
	$$
	Note that~\cite[Theorem 7.14]{liptser}, together with the condition (\ref{eq_cond_girs2}), implies that that $\mu_Z \ll \mu_W$ and $\mu_{Z^x} \ll \mu_W$, and furthermore,
	$$
	\frac{d\mu_{\xi,Z}}{d(\mu_\xi\times\mu_W)}(x,w) = \frac{d\mu_{Z^x}}{d\mu_W}(w) \quad \mu_\xi\times\mu_W \as,
	$$
	and
	$$
	\frac{d\mu_{\xi,Z}}{d(\mu_\xi\times\mu_W)}(x,W)= \frac{d\mu_{Z^x}}{d\mu_W}(W) = \exp\left( \int_0^T \gamma(t,x,W) dW(t) - \frac{1}{2} \int_0^T \gamma^2(t,x,W) dt \right),
	$$
	where $\gamma: [0,T]\times C[0,T]\to\real$ is a non-anticipative functional such that
$$
\gamma(t,x,Z^x) = \gamma(t,x,Q(\cdot,x,W)) = \E\left[ h(t,x,W) \big| \filtration{Z^x}{t} \right].
$$ Since $\mu_Z \ll \mu_W$, by \cite[Lemma 4.10]{liptser}, we have
	$$
	\frac{d\mu_{\xi,Z}}{d(\mu_\xi\times\mu_W)}(\xi,Z) = \exp\left(\int_0^T \gamma(t,\xi,Z) dZ(t) - \frac{1}{2}\int_0^T\gamma^2(t,\xi,Z) dt \right).
	$$
	
	We now claim that for each $t\in[0,T]$, $\gamma(t,\xi, Z) =\E\left[h(t,\xi,W) \big| \filtration{\xi,Z}{t} \right] \ \p\as$, which will imply the proposition. Indeed, for any bounded, $\filtration{\xi,Z}{t}$-measurable random variable $\lambda(\xi,Z)$, we have
	\begin{align*}
		\E\left[\lambda(\xi,Z) \gamma(t,\xi,Z)  \right] &= \E\left[ \lambda(\xi,Q(t,\xi,W)) \gamma(t,\xi,Q(t,\xi,W)) \right]\\
		&= \int_{C[0, T]} \int_{C[0, T]} \lambda(x, Q(t,x,w)) \gamma(t,x, Q(x,w)) \mu_W(dw) \times \mu_\xi(dx)\\
		&= \int_{C[0, T]} \int_{C[0, T]} \lambda(x, Q(t,x,w)) h(t,x,w) \mu_W(dw) \times \mu_\xi(dx)\\
		&= \E\left[\lambda(\xi,Z) h(t,\xi,W) \right],
	\end{align*}
	where the second equality follows from the independence of $\xi$ and $W$, and the third one follows from the fact that $\gamma(t,x,Q(x,W)) = \E\left[ h(t,x,W) \big| \filtration{Z^x}{t} \right]$. The proof of the claim is then completed by noting that $\gamma(t,\xi,Z)$ is $\filtration{\xi,Z}{t}$-measurable.	
	
	\section{Proof of Proposition~\ref{lemma_relative_entropy_convex}} \label{convexity-lemma}
	
	We will need the following lemma, which generalizes the well-known log-sum inequality for discrete probability distributions (see, e.g.,~\cite{CoverThomas}).		
	\begin{lemma} \label{lemma_relative_entropy_1}
		Let $f,g$ be positive measurable functions defined on a probability space $(E,\mathcal{E},\mu)$ such that
		$$
		\mu(f), \; \mu(g), \; \mu(f \log f), \; \mu(f \log g) < \infty,
		$$
		where $\mu(\cdot)$ means the integral with respect to $\mu$. Then, we have
		\begin{equation} \label{eq_relative_entropy_lemma}
			\int_E f(x) \log\frac{f(x)}{g(x)} \mu(dx) \geq \mu(f) \log \frac{\mu(f)}{\mu(g)}.
		\end{equation}
	\end{lemma}
	
	\begin{proof}
		We define two probability measures $\nu, \kappa$ as follows: for any $A \in \mathcal{E}$,
		$$
        \nu(A) := \int_A \frac{f(x)}{\mu(f)} \mu(dx) \quad \text{and} \quad \kappa(A):= \int_A \frac{g(x)}{\mu(g)} \mu(dx).
        $$
		Since $f/g >0$, we have $\nu \ll \kappa$ with $d\nu/d\kappa = (\mu(f)/\mu(g))^{-1} f/g$. Therefore, by the well-known fact that relative entropy is non-negative~\cite{ITCS}, we have
		\begin{align*}
			0  \leq \D{\nu}{\kappa} = \int_E \log \frac{d\nu}{d\kappa}(x) \nu(dx) &= \int_E \frac{d\nu}{d\mu}(x) \log \frac{d\nu}{d\kappa}(x) \mu(dx)  \\
			&= \mu(f)^{-1} \int_E f(x) \log \frac{f(x)}{g(x)} \mu(dx) - \log \frac{\mu(f)}{\mu(g)},
		\end{align*}
establishing the lemma.
	\end{proof}
	
	\begin{proof}[Proof of Lemma~\ref{lemma_relative_entropy_convex}]
		First of all, we define
$$
f(w,z) := \dfrac{d\mu_{U,Z}}{d\left(\mu_V \times \mu_Z \right)}(w,z), \mbox{ and } g(w,z) := \dfrac{d\mu_{{V,Z}}}{d\left(\mu_V \times \mu_Z \right)}(w,z).
$$
Then, by Lemma~\ref{lemma_relative_entropy_1}, we have, for each $w \in E$,
		\begin{equation} \label{eq_relative_entropy_convex}
			\int_E f(w,z) \log \frac{f(w,z)}{g(w,z)} \mu_Z(dz) \geq \mu_Z(f(w,\cdot)) \log \frac{\mu_Z(f(w,\cdot))}{\mu_Z(g(w,\cdot))},
		\end{equation}
		where
		$$
		\mu_Z(f(w,\cdot)) = \int_E\frac{d\mu_{U,Z}}{d\left(\mu_V \times \mu_Z \right)}(w,z) \mu_Z(dz) = \density{U}{V}{w}
		$$
		and
		$$
		\mu_Z(g(w,\cdot)) = \int_E\frac{d\mu_{V,Z}}{d\left(\mu_V \times \mu_Z \right)}(w,z)\mu_Z(dz) = 1.
		$$
		Finally, integrating both sides of (\ref{eq_relative_entropy_convex}) with respect to the measure $\mu_V$ yields the desired result.
	\end{proof}
	
	\section{Proof of Lemma~\ref{pp_ui}} \label{proof_pp_ui}
	
	By~\cite[Theorem 7.1]{liptser}, we have
	\begin{equation} \label{eq_cond_proof_bound}
		\hspace{-1cm} \frac{d\mu_{Y^{(n+1)}}}{d\mu_W}(Y^{(n+1)}) = \frac{1}{\E\left[\exp\left( -\int_0^T f^{(n)}(t)dY^{(n+1)}(t) + \frac{1}{2} \int_0^T \left(f^{(n)}(t)\right)^2dt \right)\Big| \filtration{Y^{(n+1)}}{T}\right]}.
	\end{equation}
	Then, applying Jensen's inequality, we establish (\ref{bound-1}) as follows:
	{\small \begin{align*}
			\hspace{-1cm} \E\left[\left(\frac{d\mu_{Y^{(n+1)}}}{d\mu_W}(Y^{(n+1)})\right)^p\right] &=\E\left[ \frac{1}{\E^p\left[\exp\left( -\int_0^T f^{(n)}(t)dY^{(n+1)}(t) + \frac{1}{2} \int_0^T \left(f^{(n)}(t)\right)^2dt \right)\Big| \filtration{Y^{(n+1)}}{T} \right]}\right]\\
			&\leq \E\left[ \exp\left( p\int_0^Tf^{(n)}(t)dY^{(n+1)}(t)-\frac{p}{2}\int_0^T\left(f^{(n)}(t)  \right)^2dt  \right)    \right]\\
			&=\E\left[ \exp\left( p\int_0^Tf^{(n)}(t)dW(t)-\frac{p^2}{2}\int_0^T\left(f^{(n)}(t)  \right)^2dt +\frac{p(p+1)}{2} \int_0^T\left(f^{(n)}(t)  \right)^2dt \right)    \right]\\
			&\leq e^{\frac{p(p+1)M}{2}}\E\left[ \exp\left( p\int_0^Tf^{(n)}(t)dW(t)-\frac{p^2}{2}\int_0^T\left(f^{(n)}(t)  \right)^2dt \right)    \right]\\
			&=e^{\frac{p(p+1)M}{2}},
	\end{align*}}
	where the last equality follows from the martingale property of the process
	$$
	\left\{\exp\left(p\int_0^tf^{(n)}(s)dW(s)-\frac{p^2}{2}\int_0^t\left(f^{(n)}(s)  \right)^2ds  \right), \ 0\leq t\leq T\right\}.
	$$
	A parallel argument can be used to establish (\ref{bound-2}).

\end{document}